\documentclass[11pt]{amsart}
\usepackage{amsmath,amssymb,latexsym,soul,cite,mathrsfs}

\usepackage{color,enumitem,graphicx}
\usepackage{dsfont}
\usepackage[colorlinks=true,urlcolor=blue,
citecolor=red,linkcolor=blue,linktocpage,pdfpagelabels,
bookmarksnumbered,bookmarksopen]{hyperref}
\usepackage[english]{babel}

\usepackage[left=2.9cm,right=2.9cm,top=2.8cm,bottom=2.8cm]{geometry}
\usepackage[hyperpageref]{backref}
\usepackage{graphicx}

\usepackage[colorinlistoftodos]{todonotes}
\makeatletter
\providecommand\@dotsep{5}
\def\listtodoname{List of Todos}
\def\listoftodos{\@starttoc{tdo}\listtodoname}
\makeatother

\numberwithin{equation}{section}

\def\R {{\rm I}\hskip -0.85mm{\rm R}}
\def\N {{\rm I}\hskip -0.85mm{\rm N}}

\newcommand{\rig}{\rightarrow}

\newtheorem{theorem}{Theorem}[section]
\newtheorem{proposition}[theorem]{Proposition}
\newtheorem{lemma}[theorem]{Lemma}

\newtheorem{remark}{Remark}

\title[On asymptotically linear elliptic problems]
{On asymptotically linear elliptic problems}

\author[J. R. S. Nascimento]{Jos\'e R. S. Nascimento}
\author[M. T. O. Pimenta]{Marcos T. O. Pimenta}
\author[J. R. Santos Jr.]{Jo\~ao R. Santos J\'unior}

\address[J.R.S. Nascimento]{\newline\indent Faculdade de Matem\'atica
\newline\indent 
Instituto de Ci\^{e}ncias Exatas e Naturais
\newline\indent 
Universidade Federal do Par\'a
\newline\indent
Avenida Augusto corr\^{e}a 01, 66075-110, Bel\'em, PA, Brazil}
\email{\href{mailto:  jrsn19@hotmail.com}{jrsn19@hotmail.com}}

\address[M. T. O. Pimenta]{\newline\indent Departamento de Matem\'atica e Computa\c{c}\~ao
\newline\indent 
Faculdade de Ci\^encias e Tecnologia
\newline\indent 
Universidade Estadual Paulista - UNESP
\newline\indent 
Rua Roberto Simonsen, 305, 19060-900, Presidente Prudente, SP, Brazil }
\email{\href{mailto: marcos.pimenta@unesp.br}{marcos.pimenta@unesp.br}}

\address[J. R. Santos Jr.]{\newline\indent Faculdade de Matem\'atica
\newline\indent 
Instituto de Ci\^{e}ncias Exatas e Naturais
\newline\indent 
Universidade Federal do Par\'a
\newline\indent
Avenida Augusto corr\^{e}a 01, 66075-110, Bel\'em, PA, Brazil}
\email{\href{mailto: joaojunior@ufpa.br }{joaojunior@ufpa.br}}

\thanks{Marcos T.O. Pimenta is partially supported by FAPESP 2021/04158-4, CNPq 303788/2018-6 and FAPDF, Brazil.  Joao R. Santos J\'unior is partially supported by FAPESP 2021/10791-1 and CNPq 313766/2021-5, Brazil.\\ Data sharing not applicable to this article as no datasets were generated or analysed during the current study.}

\begin{document}

\maketitle
\begin{abstract}
In this paper we prove the existence of a signed ground state solution in the mountain pass level for a class of asymptotically linear elliptic problems, even when the nonlinearity is just continuous in the second variable. The (strongly) resonant and non-resonant cases are discussed. A multiplicity result is also proved when $f$ is odd with respect to the second variable.
\end{abstract}
\maketitle

\section{Introduction}


We are interested in the existence of ground state and other nontrivial solutions to the following class of semilinear problems
\begin{equation}\label{P}\tag{P}
\left \{ \begin{array}{ll}
-\Delta u = f(x, u) & \mbox{in $\Omega$,}\\
u=0 & \mbox{on $\partial\Omega$,}
\end{array}\right.
\end{equation}
where $\Omega \subset\R^{N}$ is a bounded smooth domain, $N \geq 1$, $f:\Omega\times \R\to\R$ is a Carath\'eodory function which is asymptotically linear at the origin and at infinity, that is,
\begin{equation}\label{positive}
\alpha(x)=\lim_{t\to 0}2F(x, t)/t^{2} \ \mbox{and} \ \eta(x)=\lim_{|t|\to \infty}2F(x, t)/t^{2} \ \mbox{uniformly in $x\in\Omega$}
\end{equation}
are functions in $L^{\infty}(\Omega)$.

\medskip

It is well known in the literature that asymptotically linear problems can be classified as {\it resonant at infinity} (if $\lambda_{m}(\eta)=1$, for some $m\in\N$) or {\it non-resonant at infinity} (if $\lambda_{m}(\eta)\neq 1$, for all $m\in\N$), where, throughout this paper, $\lambda_{m}(\theta)$ denotes the $m$-th eigenvalue of the problem 
\begin{equation}\label{EP}
\left \{ \begin{array}{ll}
-\Delta u = \lambda \theta(x)u & \mbox{in $\Omega$,}\\
u=0 & \mbox{on $\partial\Omega$.}
\end{array}\right.
\end{equation}
In particular, observe that if $\theta(x)=\theta$ is a nontrivial constant, then $\lambda_{m}(\theta)=\lambda_{m}/\theta$, where $\lambda_{m}$ denotes the $m$-th eigenvalue of Laplacian operator with Dirichlet boundary condition. In fact, the resonant case is subdivided depending on how small at infinity is the function 
\begin{equation}\label{res1}
g(x,t) = \eta(x) t - f(x,t).
\end{equation} 
As observed in \cite{BBF}, the smaller $g$ is at infinity, the stronger resonance is. The worst situation is when, for a.e. $x \in \Omega$, 
\begin{equation}\label{res2}
\lim_{|t| \to\infty}g(x,t)= 0 \ \mbox{and} \ \lim_{|t| \to\infty} \int_0^t g(x,s)ds = \beta(x)\not\equiv \infty.
\end{equation} 
In this case, we say that problem \eqref{P} is {\it strongly resonant}. One of the very first works dealing with this situation is \cite{BBF} where, Bartolo, Benci and Fortunato show the existence of multiple solutions for strongly resonant problems in the presence of some symmetry in the autonomous nonlinearity. Their proofs are based on a deformation theorem and pseudo-index theory.

\medskip

Besides \cite{BBF}, there exist so many other works dealing with problem (\ref{P}). Without any intention to be complete, we refer the reader to some papers and references therein. The existence of solution for problem \eqref{P} was investigated under different conditions, for instance, by Ahmad \cite{Ah}, Amann and Zehnder \cite{AZ}, Ambrosetti \cite{Am}, Dancer \cite{Da}, De Figueiredo and Miyagaki \cite{DfM}, Li and Willem \cite{LW1}, Liu and Zou \cite{LZ2} and Struwe \cite{Str1}. Multiplicity results for problem \eqref{P} were also investigated by Li and Willem \cite{LW}, Liu and Zhou \cite{LZ} , Su \cite{Su} and Su and Zhao \cite{SZ}. It is essential to point out that in the great majority of previous references, the nonlinearity $f$ is assumed to be differentiable (or even $C^{1}$) in the second variable, being this assumption crucial in their arguments.

\medskip

P. Bartolo, V. Benci and D. Fortunato \cite{BBF} studied \eqref{P} when $f(x, t)=f(t)$ is a smooth function, satisfying:

\medskip

\begin{enumerate}
\item[$(BBF1)$] $\lim_{|t|\to\infty}(\lambda_{m}t-f(t))t=0$;
\item[$(BBF2)$] $\lim_{t\to\infty}\int_{-\infty}^{t}(\lambda_{m}s-f(s))ds=0$;
\item[$(BBF3)$] $\int_{-\infty}^{t}(\lambda_{m}s-f(s))ds\geq 0$, for all $t\in\R$.
\end{enumerate}

\medskip

\noindent Under these assumptions, the authors were able to prove the existence of solution to \eqref{P}. Observe that $(BBF1)$ implies that $f(t)/t\to \lambda_{m}$ ``faster'' than $t^{2}\to\infty$ as $|t|\to\infty$. More recently, Gongbao Li and Huan-Song Zhou \cite{LZ} relaxed the differentiability of $f$ by considering problem \eqref{P} under the following assumptions:

\medskip

\begin{enumerate}
\item[$(LZ1)$] $f\in C(\overline{\Omega}\times \R, \R)$;
\item[$(LZ2)$] $\lim_{t\to 0}f(x, t)/t=0$ uniformly in $x\in\Omega$;
\item[$(LZ3)$] $\lim_{|t|\to \infty}f(x, t)/t=l$ uniformly in $x\in\Omega$, where $l\in (0, \infty)$ is a constant, or $l=\infty$ and $|f(x, t)|\leq c_{1}+c_{2}|t|^{q-1}$, for some positive constants $c_1, c_{2}$ and $q\in (2, 2^{\ast})$;
\item[$(LZ4)$] $f(x, t)/t$ is nondecreasing in $t\geq 0$ for almost every $x\in\Omega$;
\item[$(LZ5)$] $f(x, -t)=-f(x, t)$ for all $(x, t)\in \Omega\times\R$.
\end{enumerate}

\medskip

\noindent By using a symmetric version of the mountain pass theorem for $C^1$-functionals, the authors were able to prove the following result:
\begin{theorem}
Assume that $f(x, t)$ satisfies conditions $(LZ1)-(LZ3)$ and $(LZ5)$, and $l>\lambda_m$, then the following hold:
\begin{enumerate}
\item[$(i)$]  If $l\in(\lambda_m, \infty)$ is not an eigenvalue of $-\Delta$ with zero Dirichlet data, then problem \eqref{P} has at least $k$ pairs of nontrivial solutions in $H_{0}^{1}(\Omega)$.
\item[$(ii)$] Suppose that condition 
\begin{equation}\label{fF}\tag{fF}
\lim_{|t|\to\infty}\left[(1/2)f(x, t)t-F(x, t)\right]=\infty
\end{equation}
is satisfied, where $F$ is the primitive of $f$. Then the conclusion of $(i)$ holds even if $l=\lambda_m$ is an eigenvalue of $-\Delta$ with zero Dirichlet data.
\item[$(iii)$] If $l=\infty$ in condition $(LZ3)$ and $(LZ4)$ holds, then problem \eqref{P} has infinitely many nontrivial solutions.
\end{enumerate}
\end{theorem}

\medskip

\noindent Let $g$ be given in \eqref{res1} and $G$ its primitive, since 
$$
(1/2)f(x, t)t-F(x, t)=G(x, t)-(1/2)t g(x, t),
$$
it follows that if $t g(x, t)=t^{2}[\eta(x)-f(x, t)/t]\to 0$ as $|t|\to\infty$, then
\begin{equation}\label{ress}
\lim_{|t|\to\infty}\left[(1/2)f(x, t)t-F(x, t)\right]=\lim_{|t|\to\infty}G(x, t)=\beta(x).
\end{equation}
Consequently, in the resonant case (that is $\lambda_m(\eta)=1$ for some $m$), the limit in \eqref{ress} determines the degree of resonance of problem \eqref{P}. For instance, in \cite{BBF} we have $\beta(x)=\beta\in\R$, thus the problem is strongly resonant. On the other hand, the result for the resonant situation in \cite{LZ} does not cover the strongly resonant case when $f(x, t)/t\to \eta(x)=l=\lambda_m$ ``faster'' than $t^{2}\to\infty$ as $|t|\to\infty$. In fact, since $(fF)$ is assumed, it follows from \eqref{ress} that $\beta(x)\equiv\infty$.

\medskip

In this paper we complement previous results. In fact, in the sequel, we suppose that $f$ satisfies the following hypotheses:

\medskip

\begin{itemize}
\item[$(f_1)$] $t\mapsto f(x, t)/|t|$ is increasing (a. e. in $\Omega$) and $\alpha^{+}, \eta^{+}\neq 0$;
\item[$(f_{2})$] $\lambda_{m}(\eta)<1<\lambda_{1}(\alpha)$, for some $m\geq 1$,   
\end{itemize}

\medskip

\noindent where $\alpha^{+}$ and $\eta^{+}$ are the positive parts of functions $\alpha$ and $\beta$ defined in \eqref{positive}. It is important to point out that, by $\alpha^{+}, \eta^{+}\neq 0$ and \cite{DeFig}, there exist the positive eigenvalues $\lambda_{m}(\eta)$ and $\lambda_{1}(\alpha)$. 

\medskip

Assuming $(f_1)-(f_2)$ we have provided existence of ground state solution and multiplicity for \eqref{P} in both cases: the non-resonant case $(NRC)$ and resonant case $(RC)$ (see Theorem \ref{teo1}), that is:

\medskip

\begin{enumerate}
\item[$(NRC)$] $\lambda_{m+k}(\eta)\neq 1$ for all $k\in\N$,
\item[$(RC)$] $\lambda_{m+k}(\eta)= 1$ for some $k\in\N$.
\end{enumerate}  

\medskip

\noindent In particular, since in the case $(RC)$ we have not assumed hypothesis $(fF)$, our result cover strongly resonant nonlinearities, even when $f(x, t)/t\to \lambda_m$ ``faster'' than $t^{2}\to\infty$ as $|t|\to\infty$, see Section \ref{sec:assump}.

\medskip

Finally, replacing assumptions $(f_{1})-(f_{2})$ by
\begin{itemize}
\item[($f_1'$)] $\sup_{t\in A}|F(., t)/t^2|_\infty<\infty$, for each bounded set $A\subset\R$,  and $\alpha^{+}, \eta^{+}\neq 0$;
\item[$(f_2')$] $\lambda_{m}(\alpha)<1<\lambda_{1}(\eta)$, for some $m\geq 1$,
\end{itemize}
we still have been able to prove the existence of multiple solution for \eqref{P}, see Theorem \ref{teo3}.

\medskip

In Theorems \ref{teo1} and \ref{teo3}, some progresses are obtained regarding the previous works. In what follows, we enumerate the main contributions: (1) Condition \eqref{fF} is not required. Instead, we are imposing condition \eqref{finito}, which is certainly weaker than \eqref{fF}. In our best knowledge, condition \eqref{finito} has not appeared in previous papers and allows us to study the existence of ground state solutions for some classes of strongly resonant problems which had not been treated. In a first moment, due to the presence of the number $\tau_{m}$, whose dependence on $f$ is not so explicit, assumption \eqref{finito} may seem difficult to be checked. In order to ensure its viability, we provide in the section \ref{sec:assump} a concrete problem \eqref{P} for which \eqref{finito} holds and \eqref{fF} is not verified; (2) Since $\alpha$, $\eta$ and $\beta$ depend on $x$ and $f$ is not differentiable, our assumptions are more general than a large part of previous papers (which usually require these functions be constants); (3) We provide an unified approach to deal concurrently the non-resonant case, the strong and the non-strong resonant cases.

\medskip

Our approach is based on the Nehari method, which consists in minimizing the energy functional $I$ over the so called Nehari manifold $\mathcal{N}$, a set which contains all the nontrivial solutions of the problem. Although this method has been carefully treated by A. Szulkin and T. Weth \cite{SW} for the case of nonlinearities which satisfies superquadraticity conditions at infinity, it is not a trivial task to apply it for problems involving asymptotically linear nonlinearities. In order to cite the main difficulties, we point out that the method consists in proving the existence of a homeomorphism $\gamma$ between $\mathcal{N}$ and a submanifold $\mathcal{M}$ of $H_{0}^{1}(\Omega)$. Despite the absence of a differentiable structure in $\mathcal{N}$, such a homeomorphism allows us to define a $C^1$-functional $\Psi$ on $\mathcal{M}$ with very useful properties. However, differently of the problem for superquadract nonlinearities, in which $\mathcal{M}$ is the unit sphere $\mathcal{S}$ of $H_{0}^{1}(\Omega)$, for asymptotically linear nonlinearities, it is not exactly clear who is the suitable manifold $\mathcal{M}$. In fact, after a careful study (see Lemma \ref{lemma1}, Propositions \ref{proposição1} and \ref{proposition2}) we are able to prove that $\mathcal{M}=\mathcal{S}_{\mathcal{A}}:=\mathcal{S}\cap\mathcal{A}$ is a noncomplete submanifold of $H_{0}^{1}(\Omega)$, where $\mathcal{A}:=\left\{u\in H_{0}^{1}(\Omega): \|u\|^{2}<\int_{\Omega}\eta(x)u^{2}dx\right\}$. This fact brings additional problems. Indeed, it is important to assure that minimizing sequences $\{u_{n}\}$ for $\Psi$ are not near the boundary of $\mathcal{S}_{\mathcal{A}}$. In \cite{SW}, this step is strongly based in the fact that $f$ has a superquadratic growth at infinity, what implies that $\{\Psi(u_{n})\}$ tends to infinity as the distance from $\{u_{n}\}$ to the boundary tends to zero. In our case, the behaviour of $\{\Psi(u_{n})\}$ at infinity, as $dist(u_{n}, \partial \mathcal{S}_{\mathcal{A}})\to 0$, is indefinite. This fact makes difficult, for example, to know how to extend $\Psi$ to $\overline{\mathcal{S}}_{\mathcal{A}}$ in order to apply the Ekeland variational principle, which is crucial to prove that $\{u_{n}\}$ can be seen as Palais-Smale sequence. 

\medskip

%

The paper is organized as follows. In Section \ref{sec:prelim} we present the variational background. In Section \ref{sec:Nehari} we deeply study the Nehari manifold and some of its topological features. In Section \ref{sec:multiplicity} we state and prove our main results. Finally, in Section \ref{sec:assump}, hypothesis \eqref{finito} is discussed in a concrete problem.

\medskip

\section{Preliminaries}\label{sec:prelim}

Our main goal in this section is to introduce some variational background for \eqref{P}. We start denoting by $I:H_{0}^{1}(\Omega)\to\R$ the energy functional associated to problem \eqref{P}, given by
$$
I(u)=\frac{1}{2}\|u\|^{2}-\int_{\Omega}F(x, u)dx,
$$
where $\displaystyle \|u\|^{2} = \int_\Omega |\nabla u|^2 dx$ and $\displaystyle  F(x, t)=\int_{0}^{t}f(x, s)ds$. It is well known that $I\in C^{1}(H_{0}^{1}(\Omega), \R)$ and
$$
I'(u)\varphi=\int_{\Omega}\nabla u\nabla \varphi dx-\int_{\Omega}f(x, u)\varphi dx.
$$
Thus, critical points of $I$ are weak solutions of \eqref{P}.

\medskip

The Nehari manifold associated to the functional $I$ is the set
$$
\mathcal{N}=\{u\in H_{0}^{1}(\Omega)\backslash\{0\}:\|u\|^{2}=\int_{\Omega}f(x, u)udx\}.
$$
Since $f$ is just a Carath\'eodory function, we cannot ensure that $\mathcal{N}$ is a smooth manifold. Moreover, $\mathcal{S}$ denotes the unit sphere in $H_{0}^{1}(\Omega)$
and
$$
\mathcal{A}:=\left\{u\in H_{0}^{1}(\Omega): \|u\|^{2}<\int_{\Omega}\eta(x)u^{2}dx\right\}.
$$ 

\medskip

Now, it is important to fix some notation. Throughout this paper we denote by $e_{j}$ a normalized (in $H_{0}^{1}(\Omega)$ norm) eigenfunction associated to $\lambda_{j}(\eta)$ and use the symbol $[u\neq 0]$ to denote the set $\{x\in\Omega: u(x)\neq 0\}$. Moreover, $|A|$ will always denote the Lebesgue measure of a measurable set $A\subset\R^{N}$, $S(\Omega)$ and $|\theta|_{\infty}$ denote, respectively, the best constant of the continuous embedding from $H_{0}^{1}(\Omega)$ into $L^{2^{\ast}}(\Omega)$ and the $L^{\infty}$-norm of a function $\theta$, $\chi(\theta)=\sum_{j=1}^{m}dim V_{\lambda_{j}(\theta)}$ is the sum of the dimensions of the first $m$ eigenspaces $V_{\lambda_{j}(\theta)}$ associated to \eqref{EP}. Finally, $\mathcal{S}_{\chi(\theta)}$ denotes the unit sphere of $\oplus_{k=1}^{m}V_{\lambda_{k}(\theta)}$ and $\theta^{+}$ denotes the positive part of a function $\theta$.

\medskip

\begin{lemma}\label{lemma1}
Suppose that $f$ satisfies $(f_{1})-(f_{2})$. Then, the following claims hold:
\begin{enumerate}
\item[$(i)$] The set $\mathcal{A}$ is open and nonempty;
\item[$(ii)$] $\partial \mathcal{A}=\{u\in H_{0}^{1}(\Omega): \|u\|^{2}=\int_{\Omega}\eta(x)u^{2}dx\}$;
\item[$(iii)$] $\mathcal{A}^{c}=\{u\in H_{0}^{1}(\Omega): \|u\|^{2}\geq\int_{\Omega}\eta(x)u^{2}dx\}$;
\item[$(iv)$] $\mathcal{N}\subset \mathcal{A}$;
\item[$(v)$] $\mathcal{S}\cap\mathcal{A}\neq\emptyset$.
\end{enumerate}
\end{lemma}
\begin{proof} ($i$) By $(f_{2})$, if $u$ is an eigenfunction associated to $\lambda_{j}(\eta)$, for some $j\in\{1, \ldots, m\}$, then $u$ belongs to $\mathcal{A}$. Moreover, $\mathcal{A}=\varphi^{-1}(-\infty, 0)$, where $\varphi:H_{0}^{1}(\Omega)\to \mathbb{R}$ is the continuous function $\varphi(u)=\|u\|^{2}-\int_{\Omega}\eta(x)u^{2}dx$. Items ($ii$)-($iii$) are immediate consequences of the definition of $\mathcal{A}$.

\medskip

 ($iv$) If $u\in\mathcal{N}$ then, 
$$
\|u\|^{2}=\int_{[u\neq 0]}\left[\frac{f(x, u)}{u}\right]u^{2}dx.
$$
By $(f_{1})$, we conclude  
$$
\|u\|^{2}<\int_{\Omega}\eta(x)u^{2}dx.
$$
Showing that $u\in \mathcal{A}$.

\medskip

($v$) It is enough to choose an eigenfunction $e_{j}$ associated to $\lambda_{j}(\eta)$ and normalized in $H_{0}^{1}(\Omega)$,  whatever $j\in\{1, \ldots, m\}$. For sure, we have $e_{j}\in\mathcal{S}\cap\mathcal{A}$.
\end{proof}

\medskip

In what follows, we will denote $\mathcal{S}_{\mathcal{A}}:=\mathcal{S}\cap\mathcal{A}$. Since $\mathcal{S}$ is a $C^{1}$-manifold of $H^{1}_{0}(\Omega)$ and, by Lemma $\ref{lemma1}$, $\mathcal{A}$ is open set of $H^{1}_{0}(\Omega)$ whose boundary has intersection with $\mathcal{S}$, it follows that $\mathcal{S}_{\mathcal{A}}$ is a noncomplete $C^{1}$-manifold of $H^{1}_{0}(\Omega)$. Moreover, from $(ii)$ and $(iii)$, it is clear that $\partial \mathcal{S}_{\mathcal{A}}=\{u\in \mathcal{S}: 1=\int_{\Omega}\eta(x)u^{2}dx\}$ and $\mathcal{S}_{\mathcal{A}}^{c}=\{u\in \mathcal{S}: 1\geq\int_{\Omega}\eta(x)u^{2}dx\}$.

\medskip

The following property of functions in $\partial\mathcal{S}_{\mathcal{A}}$ plays an important role in the existence of solution for \eqref{P}, see Lemmas \ref{level} and \ref{minimax}.

\medskip

\begin{lemma}\label{limita}
The following inequality holds true:
$$
\inf_{u\in \partial\mathcal{S}_{\mathcal{A}}}|[u\neq 0]|\geq \left( S(\Omega)/|\eta|_{\infty}\right)^{N/2}.
$$
\end{lemma}

\begin{proof}
By using H\"{o}lder inequality, it follows that, for each $u\in \partial\mathcal{S}_{\mathcal{A}}$
$$
1\leq |\eta|_{\infty}\int_{[u\neq 0]}u^{2}dx\leq |\eta|_{\infty}|u|_{2^{\ast}}^{2}|[u\neq 0]|^{2/N}.
$$
By continuous Sobolev embedding from $H_{0}^{1}(\Omega)$ into $L^{2^{\ast}}(\Omega)$, 
$$
1\leq |\eta|_{\infty}(1/S(\Omega))|[u\neq 0]|^{2/N}.
$$
Therefore,
$$
|[u\neq 0]|\geq (S(\Omega)/|\eta|_{\infty})^{N/2}, \ \forall \ u\in \partial\mathcal{S}_{\mathcal{A}}.
$$
The result is proved.
\end{proof}

\medskip

Next lemma provides some consequences of hypothesis $(f_{1})$ which will be useful later on.

\medskip

\begin{lemma}\label{lemma2}
Suppose that $(f_{1})$ holds. Then, a.e. in $\Omega$,
\begin{enumerate}
\item[$(i)$] $t\mapsto (1/2)f(x, t)t-F(x, t)$ is increasing in $(0,\infty)$ and decreasing in $(-\infty, 0)$;
\item[$(ii)$] $t\mapsto F(x, t)/t^{2}$ is increasing in $(0,\infty)$ and decreasing in $(-\infty, 0)$;
\item[$(iii)$] $f(x, t)/t>2F(x, t)/t^{2}$ for all $t\in \R\backslash\{0\}$.
\end{enumerate}
\end{lemma}

\medskip

\begin{proof}
$(i)$ Without loss of generality, we can suppose $t_{1}>t_{2}>0$. Then, a.e. in $\Omega$,
\begin{eqnarray*}
\frac{1}{2}f(x, t_{1})t_{1}-F(x, t_{1}) &=&\frac{1}{2}f(x, t_{1})t_{1}-F(x, t_{2})-\int_{t_{2}}^{t_{1}}\left[\frac{f(x, s)}{s}\right]sds\\
&>&\frac{1}{2}f(x, t_{1})t_{1}-F(x, t_{2})-\frac{f(x, t_{1})}{t_{1}}\int_{t_{2}}^{t_{1}}sds\\
&=& \frac{1}{2}f(x, t_{1})t_{1}-F(x, t_{2})-\frac{f(x, t_{1})}{t_{1}}\frac{(t_{1}^{2}-t_{2}^{2})}{2}\\
&=&\frac{f(x, t_{1})}{t_{1}}\frac{t_{2}^{2}}{2}-F(x, t_{2})\\
&> &\frac{1}{2}f(x, t_{2})t_{2}-F(x, t_{2}),
\end{eqnarray*}
where it was used $(f_{1})$ in the last two inequalities. The other case is analogous. Items $(ii)$ and $(iii)$ follows from $(i)$.
\end{proof}

\medskip

From Lemma \ref{lemma2}$(i)$, the maps $\beta:\Omega\to(0,\infty]$ given by
\begin{equation}\label{beta}
\beta(x)=\lim_{|t|\to\infty}\left[(1/2)f(x, t)t-F(x, t)\right]
\end{equation}
is well defined, and, it might not be integrable.

\section{Topological aspects of the Nehari manifold}\label{sec:Nehari}

The main goal of this section is to study some topological features of the Nehari manifold under hypotheses $(f_1)-(f_2)$ and the behaviour of the energy functional $I$ on $\mathcal{N}$.

\medskip

\begin{proposition}\label{proposição1}
Suppose that $f$ satisfies $(f_1)-(f_{2})$ and let $h_{u}:[0,\infty)\rightarrow \R$ be defined by
$h_{u}(t)=I(tu)$.

\medskip

\begin{enumerate}
\item[$(i)$] For each $u\in \mathcal{A}$, there exists a unique $t_{u}>0$ such that $h_{u}'(t)>0$ in $(0,t_{u})$, $h_{u}'(t_{u})=0$ and $h_{u}'(t)<0$ in $(t_{u}, \infty)$. Moreover, $tu\in \mathcal{N}$ if, and only if, $t=t_{u}$;

\item[$(ii)$] for each $u\in \mathcal{A}^{c}$, $h_{u}'(t)>0$ for all $t\in (0, \infty)$.
\end{enumerate}
\end{proposition}

\begin{proof} 

\medskip

($i$) First observe that $h_{u}(0)=0$. Moreover, for each $u\in \mathcal{A}$, we have
\begin{equation}
\frac{h_{u}(t)}{t^{2}}= \frac{1}{2}\|u\|^{2}-\int_{[u\neq 0]}\left[\frac{F(x, tu)}{(tu)^{2}}\right]u^{2}dx.
\end{equation}
Thus, from $(f_{1})-(f_{2})$, L'Hospital rule and Lebesgue Dominated Convergence Theorem, it follows that
$$
\lim_{t\to 0}\frac{h_{u}(t)}{t^{2}}=\frac{1}{2}\left(\|u\|^{2}-\int_{\Omega}\alpha(x)u^{2}dx\right)>0
$$
and
$$
\lim_{t\to \infty}\frac{h_{u}(t)}{t^{2}}=\frac{1}{2}\left(\|u\|^{2}-\int_{\Omega}\eta(x)u^{2}dx\right)<0.
$$
Showing that 
$$
h_{u}(t)=\frac{h_{u}(t)}{t^{2}}t^{2}
$$
is positive for $t$ small and
$$
\lim_{t\to \infty}h_{u}(t)=\lim_{t\to \infty}\frac{h_{u}(t)}{t^{2}}t^{2}=-\infty.
$$

\medskip

Since $h_{u}$ is a continuous function, previous arguments imply that there exists a global maximum point $t_{u}>0$ of $h_{u}$. Now, we are going to show that $t_{u}$ is the unique critical point of $h_{u}$. In fact, supposing that there exist $t_{1}>t_{2}>0$ such that $h_{u}'(t_{1})=h_{u}'(t_{2})=0$, we obtain
$$
0=\int_{[u\neq 0]}\left[\frac{f(x, t_{1}u)}{t_{1}u}-\frac{f(x, t_{2}u)}{t_{2}u}\right]u^{2}dx,
$$
and, by $(f_{1})$, $t_{1}=t_{2}$. The result follows.

\medskip

($ii$) If $u\in \mathcal{A}^{c}$, then $\|u\|^{2}\geq \int_{\Omega}\eta(x)u^{2}dx$. Thus, it follows from $(f_{1})$ that
$$
\frac{h_{u}'(t)}{t}=\|u\|^{2}-\int_{[u\neq 0]}\frac{f(x, tu)}{tu}u^{2}dx\geq \int_{[u\neq 0]}\left[\eta(x)-\frac{f(x, tu)}{tu}\right]u^{2}dx>0, \ \forall \ t>0.
$$ 
Consequently, $h_{u}'(t)=t(h_{u}'(t)/t)>0$ for all $t\in (0,\infty)$.
\end{proof}

\medskip

\begin{remark}\label{rem1}
It is an immediate consequence of previous proposition that, for each $u\in \mathcal{A}$ and $s\in (0, \infty)$, $t_{su}=t_{u}/s$. Moreover, it is clear that $u\in \mathcal{N}$ if, and only if, $t_{u}=1$.
\end{remark}

\medskip

\begin{proposition}\label{proposition2}
Suppose that $f$ satisfies $(f_{1})-(f_{2})$. Then, the following claims hold:

\medskip

\begin{enumerate}
\item[$(A_{1})$] $\tau_{m}:=\inf_{u\in \mathcal{S}_{\chi(\eta)}}t_{u}>0$;

\item[$(A_{2})$] $\zeta_\mathcal{W}:=\max_{u\in \mathcal{W}}t_{u}<\infty$, for all compact set $\mathcal{W}\subset \mathcal{S}_{\mathcal{A}}$;

\item[$(A_{3})$] Map
$\widehat{m}:\mathcal{A}\rightarrow \mathcal{N}$ given by
$\widehat{m}(u)=t_{u}u$ is continuous and
$m:=\widehat{m}_{|_{\mathcal{S}_{\mathcal{A}}}}$ is a homeomorphism between
$\mathcal{S}_{\mathcal{A}}$ and $\mathcal{N}$. Moreover, $m^{-1}(u)=u/\|u\|$.
\end{enumerate}
\end{proposition}

\medskip

\begin{proof}
($A_{1}$) It is a consequence of the proof of Lemma \ref{lemma1}$(v)$ that $\mathcal{S}_{\chi(\eta)}\subset \mathcal{S}_{\mathcal{A}}$. Thus, suppose that there exists $\{u_{n}\}\subset \mathcal{S}_{\mathcal{A}}$ such that $t_{n}:=t_{u_{n}}\to 0$. In this case, we get $u\in H_{0}^{1}(\Omega)$ such that $u_{n}\rightharpoonup u$ in $H_{0}^{1}(\Omega)$. It follows from $(f_1)$ and Lebesgue Dominated Convergence Theorem, that
\begin{equation}\label{nonzero}
\int_{\Omega}\left[\frac{f(x, t_{n}u_{n})}{t_{n}u_{n}}\right]\chi_{[u_{n}\neq 0]}u_{n}^{2} dx\to \int_{\Omega}\alpha(x)u^{2}dx.
\end{equation}

\medskip

From Proposition \ref{proposição1}, we have
\begin{equation}\label{toma}
1=\int_{\Omega}\left[\frac{f(x, t_{n}u_{n})}{t_{n}u_{n}}\right]\chi_{[u_{n}\neq 0]}u_{n}^{2} dx, \ \forall \ n\in\N.
\end{equation}
Now, by 
$$
\chi_{[u_{n}\neq 0]}(x)\to 1 \ \mbox{a.e. in $[u\neq 0]$ and $\chi_{[u_{n}\neq 0]}(x)\to 0$ a.e. in $[u=0]$},
$$
and \eqref{nonzero}, passing to the limit as $n$ tends to infinity in \eqref{toma}, we get
$$
1=\int_{\Omega}\alpha(x)u^{2}dx.
$$
If $\alpha=0$, we have a clear contradiction. Otherwise, the inequality
$$
1\leq (1/\lambda_{1}(\alpha))\|u\|^{2}\leq 1/\lambda_{1}(\alpha),
$$
contradicts $(f_{2})$.

\medskip

($A_{2}$) Suppose that there exists $\{u_{n}\}\subset \mathcal{W}$ such that $t_{n}:=t_{u_{n}}\to \infty$. Since $\mathcal{W}$ is compact, passing to a subsequence, we obtain $u\in \mathcal{W}$ such that $u_{n}\to u$ in $H_{0}^{1}(\Omega)$. Since
$$
1=\|u_{n}\|^{2}=\int_{\Omega}\left[\frac{f(x, t_{n}u_{n})}{t_{n}u_{n}}\right]\chi_{[u_{n}\neq 0]}u_{n}^{2} dx, \ \forall \ n\in\N,
$$
with
$$
\int_{\Omega}\left[\frac{f(x, t_{n}u_{n})}{t_{n}u_{n}}\right]\chi_{[u_{n}\neq 0]}u_{n}^{2} dx= \int_{[u\neq 0]}\left[\frac{f(x, t_{n}u_{n})}{t_{n}u_{n}}\right]\chi_{[u_{n}\neq 0]}u_{n}^{2} dx+\int_{[u=0]}\left[\frac{f(x, t_{n}u_{n})}{t_{n}u_{n}}\right]\chi_{[u_{n}\neq 0]}u_{n}^{2} dx,
$$
where
$$
\chi_{[u_{n}\neq 0]}(x)\to 1 \ \mbox{a.e. in $[u\neq 0]$ and $\chi_{[u_{n}\neq 0]}(x)\to 0$ a.e. in $[u=0]$},
$$
passing to the lower limit as $n\to\infty$, it follows from $(f_1)$ and Lebesgue Dominated Convergence Theorem, that
$$
1=\|u\|^{2}\geq \int_{\Omega}\eta(x)u^{2}dx.
$$
Last inequality implies that $u\in \mathcal{S}_{\mathcal{A}}^{c}$, leading us to a contradiction, since $u\in \mathcal{W}\subset\mathcal{S}_{\mathcal{A}}$.

\medskip

($A_{3}$) We first show that $\widehat{m}$ is continuous. Let $\{u_{n}\}\subset \mathcal{A}$ and $u\in\mathcal{A}$, be such that $u_{n}\to u$ in $H_{0}^{1}(\Omega)$. From Remark \ref{rem1} ($\widehat{m}(tw)=\widehat{m}(w)$ for all $w\in \mathcal{A}$ and $t>0$), we can assume, without loss fo generality, that $\{u_{n}\}\subset S_{\mathcal{A}}$. Thus,

\medskip

\begin{equation}\label{conv1}
t_{n}=t_{n}\|u_{n}\|^{2}=\int_{\Omega}f(x, t_{n}u_{n})u_{n}dx,
\end{equation}
where $t_{n}:=t_{u_{n}}$. From ($A_{1}$) and $(A_{2})$, it follows that, up to a subsequence, $t_{n}\to t>0$. Thence, passing to the limit as $n\to\infty$ in \eqref{conv1}, we have
$$
t=t\|u\|^{2}=\int_{\Omega}f(x, tu)u dx,
$$
showing that $\widehat{m}(u_{n})=t_{n}u_{n}\to tu=\widehat{m}(u)$. The second part  of ($A_{3}$) is immediate.
\end{proof}

\medskip

\begin{lemma}\label{lemma3}
Suppose that $(f_1)-(f_2)$ holds. Then, $I(u)>0$, for all $u\in\mathcal{N}$.
\end{lemma}

\medskip

\begin{proof}
For any $u\in \mathcal{S}_{\mathcal{A}}$, we get
$$
I(t_{u}u)=\int_{\Omega}\left[\frac{1}{2}\frac{f(x, t_{u}u)}{t_{u}u}-\frac{F(x, t_{u}u)}{(t_{u}u)^{2}}\right](t_u u)^{2}dx.
$$
The result follows from Lemma \ref{lemma2}$(iii)$.
\end{proof}

\medskip

In what follows, let us consider the maps $\widehat{\Psi}:\mathcal{A}\rightarrow\mathbb{R} \ \mbox{and} \ \Psi:\mathcal{S}_{\mathcal{A}}\rightarrow \mathbb{R}$, given by 
$$
\widehat{\Psi}(u)=I(\widehat{m}(u)) \ \mbox{and} \ \Psi:=\widehat{\Psi}_{|_{\mathcal{S_{\mathcal{A}}}}}.
$$ 
These maps will be very important in our arguments mainly because of their properties, which will be presented in the next result. The proof of such a result is a consequence of Proposition \ref{proposition2} and the details can be found in \cite{SW}. 

\medskip



\begin{proposition}\label{proposition3}
Suppose that $f$ satisfies $(f_{1})-(f_{2})$. Then,
\begin{enumerate}
\item[$(i)$] $\widehat{\Psi}\in C^{1}(\mathcal{A}, \mathbb{R})$ and
$$
\widehat{\Psi}'(u)v=\frac{\|\widehat{m}(u)\|}{\|u\|}I'(\widehat{m}(u))v,
\ \forall u\in \mathcal{A} \ \mbox{and} \ \forall v\in H_{0}^{1}(\Omega).
$$

\item[$(ii)$] $\Psi\in C^{1}(\mathcal{S}_{\mathcal{A}}, \mathbb{R})$ and
$$
\Psi'(u)v=\|m(u)\|I'(m(u))v, \ \forall v\in T_{u}\mathcal{S}_{\mathcal{A}}.
$$

\item[$(iii)$] If $\{u_{n}\}$ is a $(PS)_{c}$ sequence for $\Psi$ then  $\{m(u_{n})\}$ is a $(PS)_{c}$ sequence for $I$.
If $\{u_{n}\}\subset \mathcal{N}$ is a bounded $(PS)_{c}$ sequence for $I$ then $\{m^{-1}(u_{n})\}$ is a $(PS)_{c}$ sequence for $\Psi$.

\item[$(iv)$] $u$ is a critical point of $\Psi$ if, and only if, $m(u)$ is a nontrivial critical point of $I$. Moreover,
$$
c_{\mathcal{N}}:=\inf_{u\in \mathcal{N}}I(u)=\inf_{u\in \mathcal{A}}\max_{t>0}I(tu)=\inf_{u\in\mathcal{S}_{\mathcal{A}}}\max_{t>0}I(tu)=\inf_{u\in\mathcal{S}_{\mathcal{A}}}\Psi(u).
$$
\end{enumerate}
\end{proposition}

\medskip

\begin{remark}\label{rem2}
It is a consequence of Lemma \ref{lemma3} that $c_{\mathcal{N}}\geq 0$. Moreover, if $c_{\mathcal{N}}$ is achieved then it is positive.
\end{remark} 

\medskip

Since $\mathcal{S}_{\mathcal{A}}$ can be non-complete, we need to be careful about the behaviour of minimizing sequences for $\Psi$ near the boundary. Next result helps us in this direction.

\medskip

\begin{proposition}\label{main1}
Suppose that $(f_{1})-(f_{2})$ hold. If $\{u_{n}\}\subset \mathcal{S}_{\mathcal{A}}$ is such that $dist(u_{n}, \partial \mathcal{S}_{\mathcal{A}})\to 0$, then there exists $u\in H_{0}^{1}(\Omega)\backslash\{0\}$ such that $u_{n}\rightharpoonup u$ in $H_{0}^{1}(\Omega)$, $t_{u_{n}}\to\infty$ and 
$$
\liminf\limits_{n\to\infty}\Psi(u_{n})\geq\int_{[u\neq 0]}\beta(x) dx,
$$ 
where the maps $\beta$ was defined in \eqref{beta}. 
\end{proposition}

\medskip

\begin{proof}

Since $\{u_{n}\}\subset \mathcal{S}_{\mathcal{A}}$ is bounded, up to a subsequence, there exists $u\in H_{0}^{1}(\Omega)$ with $u_{n}\rightharpoonup u$ in $H_{0}^{1}(\Omega)$. On the other hand, from $dist(u_{n}, \partial \mathcal{S}_{\mathcal{A}})\to 0$, there exists $\{z_{n}\}\subset \partial\mathcal{S}_{\mathcal{A}}$ such that $\|u_{n}-z_{n}\|\to 0$ as $n\to\infty$. Thus,
\begin{eqnarray*}
\left| \int_{\Omega}\eta(x)u_{n}^{2}dx-1\right|&=&\left| \int_{\Omega}\eta(x)(u_{n}^{2}-z_{n}^{2})dx\right|\\
&\leq& |\eta|_{\infty}|u_{n}+z_{n}|_{2}|u_{n}-z_{n}|_{2}\\
&\leq&  (2|\eta|_{\infty}/\lambda_{1})\|u_{n}-z_{n}\|.
\end{eqnarray*}
Therefore, 
$$
\int_{\Omega}\eta(x)u_{n}^{2}dx\to 1.
$$

\medskip

By using compact embedding from $H_{0}^{1}(\Omega)$ in $L^{2}(\Omega)$, it follows that
\begin{equation}\label{equal}
1=\int_{\Omega}\eta(x)u^{2}dx.
\end{equation}
Thus
\begin{equation}\label{notnull}
u\neq 0.
\end{equation} 
Suppose by contradiction that, for some subsequence, $\{t_{u_{n}}\}$ is bounded. In this case, passing again to a subsequence, there exists $t_{0}> 0$ (see Proposition \ref{proposition2}$(A_{1})$) such that 
\begin{equation}\label{conv2}
t_{u_{n}}\to t_{0}.
\end{equation} 

\medskip

It follows from \eqref{notnull}, \eqref{conv2} and
$$
t_{u_n}=\int_{\Omega}f(x, t_{u_{n}}u_{n})u_{n} dx, \ \forall \ n\in\N,
$$
that
$$
1=\int_{[u\neq 0]}\frac{f(x, t_{0}u)}{t_{0}u}u^{2}dx.
$$
Combining last equality and $(f_{1})$, we have
\begin{equation}\label{issoai}
1<\int_{\Omega}\eta(x)u^{2}dx.
\end{equation}
Clearly, \eqref{equal} and \eqref{issoai} contradicts each other, showing that $t_{u_{n}}\to\infty$. Since the behaviour at infinity of $\{t_{u_{n}}u_{n}\}$ is indeterminated on $[u=0]$, in next inequality, we restrict our arguments to the set $[u\neq 0]$. In fact, it follows from Lemma \ref{lemma2}, that

\medskip

\begin{eqnarray*}
\liminf\limits_{n\to\infty}\Psi(u_{n})&=& \liminf\limits_{n\to\infty}\int_{\Omega}\left[ \frac{1}{2}f(x, t_{u_{n}}u_{n})t_{u_{n}}u_{n}-F(x, t_{u_{n}}u_{n})\right]dx\\
&\geq &\liminf\limits_{n\to\infty}\int_{[u\neq 0]}\left[ \frac{1}{2}f(x, t_{u_{n}}u_{n})t_{u_{n}}u_{n}-F(x, t_{u_{n}}u_{n})\right]dx\\
&=& \int_{[u\neq 0]}\beta(x) dx.
\end{eqnarray*}
\end{proof}

\medskip

The following hypothesis is certainly weaker than ($fF$) and will be considered in our next result
\begin{equation}\label{finito}\tag{$\beta$}
\inf\operatorname{ess} \limits_{x\in\Omega}\beta(x)> \frac{|\eta|_{\infty}^{N/2} \tau_{m}^{2}}{2\lambda_{1}(\eta-\alpha)S(\Omega)^{N/2}},
\end{equation}
where $\tau_{m}$ is defined in Proposition \ref{proposition2}$(i)$.

\medskip

\begin{lemma}\label{level}
Suppose that $f$ satisfies $(f_{1})-(f_{2})$ and \eqref{finito}. Then
$$
c_{\mathcal{N}}<\inf_{u\in\partial\mathcal{S}_{\mathcal{A}}}\int_{[u\neq 0]}\beta(x)dx.
$$ 
\end{lemma}

\begin{proof}

It follows from $(f_{1})$ that, for each $u\in\mathcal{S}_{\mathcal{A}}$,
\begin{eqnarray}\nonumber\label{uuuuu}
c_{\mathcal{N}}\leq\Psi(u)&=&\int_{\Omega}\left[ \frac{1}{2}f(x, t_{u}u)t_{u}u-F(x, t_{u}u)\right]dx\\ \nonumber
&<&(1/2)\int_{\Omega}[\eta(x)-\alpha(x)]m(u)^{2} dx\\
&\leq & [1/2\lambda_{1}(\eta-\alpha)]t_{u}^{2}.
\end{eqnarray}

\medskip

On the other hand, by \eqref{limita}, for each $u\in\partial\mathcal{S}_{\mathcal{A}}$
\begin{equation}\label{vvvvv}
\int_{[u\neq 0]}\beta(x) dx\geq |[u\neq 0]|\inf\operatorname{ess} \limits_{x\in\Omega}\beta(x)\geq \left( S(\Omega)/|\eta|_{\infty}\right)^{N/2}\inf\operatorname{ess} \limits_{x\in\Omega}\beta(x).
\end{equation}
The result follows now from \eqref{finito}, \eqref{uuuuu} and \eqref{vvvvv}.
\end{proof}

\medskip

Based in previous results, in the next illustration, we try to give an idea about some possible topological configuration of the sets $\mathcal{N}$, $\mathcal{A}$ and $S_\mathcal{A}$. The set $\mathcal{A}$ is represented as the nonempty interior of a ``cone'' in $H_{0}^{1}(\Omega)$ which intersects the unit sphere in the set $S_\mathcal{A}$ and contains the Nehari set $\mathcal{N}$. As stated in Propositions \ref{proposition2} and \ref{main1}, the Nehari set is homeomorphic to $S_\mathcal{A}$, unbounded and asymptote the boundary of $\mathcal{A}$ at infinity.

\newpage

\begin{figure}[h!]
\centering
\includegraphics[width=8cm, height=11cm]{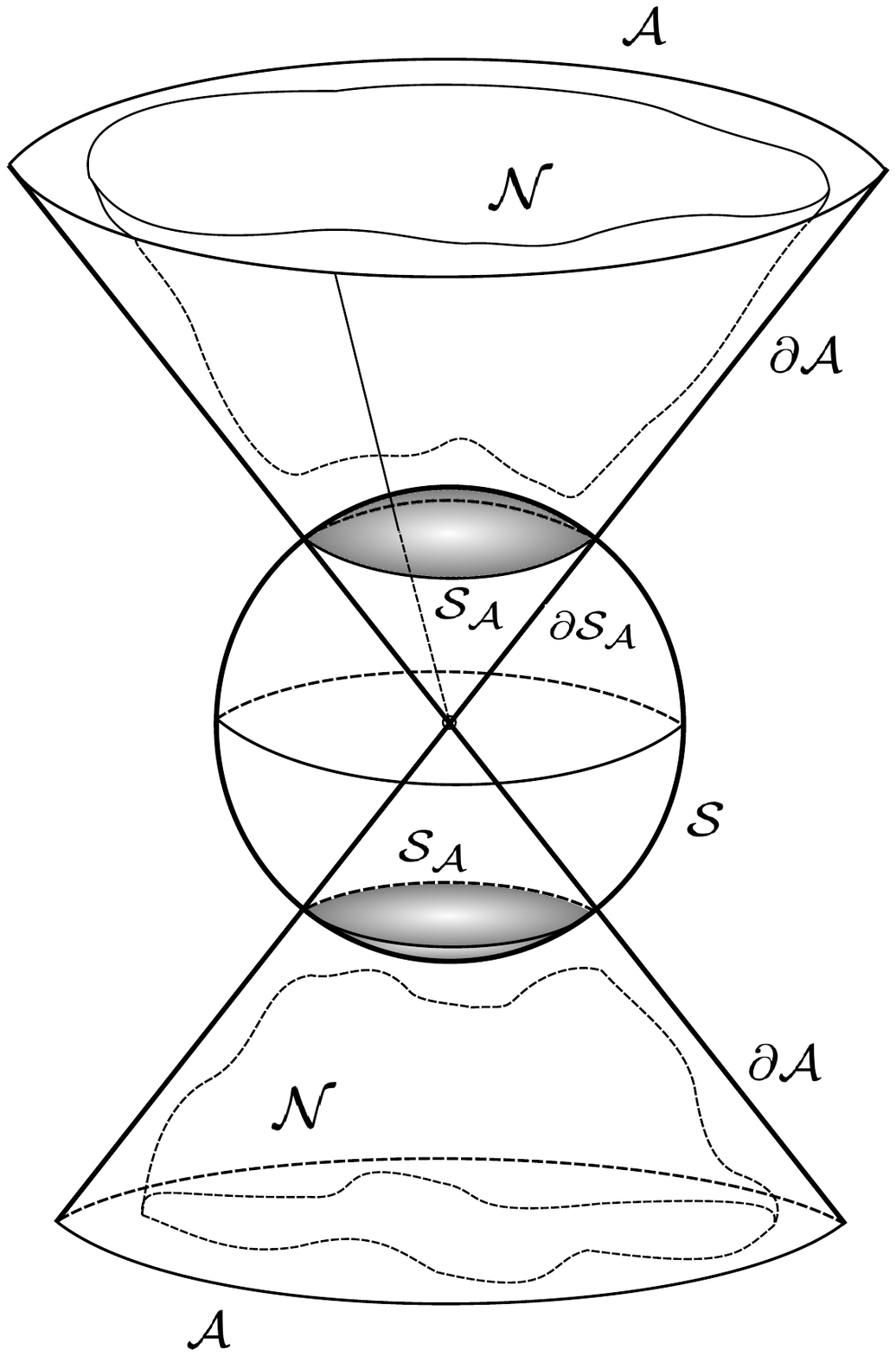}
\caption{The Nehari set $\mathcal{N}$ and the manifold $S_\mathcal{A}$}
\label{Nehari}
\end{figure}

\medskip

Before stating next result, we remember hypotheses $(NRC)$ and $(RC)$ stated in the introduction. 

\medskip

\begin{enumerate}
\item[$(NRC)$] $\lambda_{m+k}(\eta)\neq 1$ for all $k\in\N$;

\item[$(RC)$] $\lambda_{m+k}(\eta)= 1$ for some $k\in\N$.
\end{enumerate}

\medskip

\begin{proposition}\label{main2}
Suppose that $f$ satisfies $(f_{1})-(f_{2})$.
\begin{enumerate}
\item[$(i)$] If $(NRC)$ holds, then $\Psi$ satisfies the $(PS)_{c}$ condition in $\mathcal{S}_{\mathcal{A}}$, for all $c\geq c_{\mathcal{N}}$;
\item[$(ii)$] If $(RC)$ and \eqref{finito} hold, then $\Psi$ satisfies the $(PS)_{c}$ condition in $\mathcal{S}_{\mathcal{A}}$, for all $c\in [c_{\mathcal{N}}, \inf_{u\in \partial\mathcal{S}_{\mathcal{A}}}\int_{[u\neq 0]}\beta(x) dx)$.
\end{enumerate}
\end{proposition}

\begin{proof} 
By Proposition \ref{proposition2}$(A_{3})$ and Proposition \ref{proposition3}$(iii)$, it is enough to show that $I$ satisfies the $(PS)_{c}$ condition on $\mathcal{N}$ for $c\in [c_{\mathcal{N}}, \inf_{u\in\partial\mathcal{S}_{\mathcal{A}}}\int_{[u\neq 0]}\beta(x) dx)$. For this, let $\{u_{n}\}\subset \mathcal{N}$ be a $(PS)_{c}$ sequence for the functional $I$. We are going to prove that $\{u_{n}\}$ is bounded in $H_{0}^{1}(\Omega)$. In fact, suppose by contradiction that, up to a subsequence, $\|u_{n}\|\to\infty$. Define $v_{n}:=u_{n}/\|u_{n}\|=m^{-1}(u_{n})\in \mathcal{S}_{\mathcal{A}}$. Thus $\{v_{n}\}$ is bounded in $H_{0}^{1}(\Omega)$ and 

\begin{equation}\label{bound}
\Psi(v_{n})\to c.
\end{equation} 

\medskip

Consequently, there exists $v\in H_{0}^{1}(\Omega)$ such that 
\begin{equation}\label{weak}
v_{n}\rightharpoonup v \ \mbox{in \ $H_{0}^{1}(\Omega)$}.
\end{equation}

\medskip

Suppose $v=0$. Since $\{\Psi(v_{n})\}$ is bounded, it follows that there exists $C>0$ such that
\begin{equation}\label{6000}
C>\Psi(v_{n})=I(t_{v_{n}}v_{n})\geq I(tv_{n})=(1/2)t^2-\int_{\Omega}F(x, tv_{n})dx, \ \forall \ t>0.
\end{equation}
From $(f_{1})-(f_{2})$ and compact embedding, passing to the limit as $n\to\infty$ in $(\ref{6000})$, we get 
$$
C\geq (1/2)t^{2}, \ \forall \ t>0,
$$
a clear contradiction. Thereby, we conclude that 
\begin{equation}\label{notnull2}
v\neq 0.
\end{equation} 

\medskip

Now, since $\{u_{n}\}\subset\mathcal{N}$ is a $(PS)_{c}$ sequence for $I$, we get
$$
o_{n}(1)+\int_{\Omega}\nabla u_{n}\nabla w dx=\int_{\Omega}f(x, u_{n})w dx, \ \forall \ w\in H_{0}^{1}(\Omega).
$$
Dividing last equality by $\|u_{n}\|$, we have
\begin{eqnarray}\nonumber\label{cuenta}
&&o_{n}(1)+\int_{\Omega}\nabla v_{n}\nabla w dx=\\
&&\int_{[v\neq 0]}\left[\frac{f(x, \|u_{n}\|v_{n})}{\|u_{n}\|v_{n}}\right]\chi_{[v_{n}\neq 0]}(x)v_{n}w dx+\int_{[v= 0]}\left[\frac{f(x, \|u_{n}\|v_{n})}{\|u_{n}\|v_{n}}\right]\chi_{[v_{n}\neq 0]}(x)v_{n}w dx.
\end{eqnarray}

\medskip

By $(f_{1})$, \eqref{weak} and Lebesgue Dominated Convergence Theorem, it follows that
\begin{equation}\label{ahora}
\int_{[v= 0]}\left[\frac{f(x, \|u_{n}\|v_{n})}{\|u_{n}\|v_{n}}\right]\chi_{[v_{n}\neq 0]}(x)v_{n}w dx\to 0.
\end{equation}

\medskip

On the other hand, by \eqref{weak} and Lebesgue Dominated Convergence Theorem, we get
\begin{equation}\label{salio}
\int_{[v\neq 0]}\left[\frac{f(x, \|u_{n}\|v_{n})}{\|u_{n}\|v_{n}}\right]\chi_{[v_{n}\neq 0]}(x)v_{n}w dx\to \int_{\Omega}\eta(x)vw dx.
\end{equation}
It follows from \eqref{ahora} and \eqref{salio} that, passing to the limit as $n\to\infty$ in \eqref{cuenta}, we obtain
\begin{equation}\label{eigen}
\int_{\Omega}\nabla v\nabla w dx=\int_{\Omega}\eta(x)vw dx, \ \forall \ w\in H_{0}^{1}(\Omega).
\end{equation}

\medskip

Now we have to consider two cases:

\medskip

$(i)$ If $\lambda_{m+k}(\eta)\neq 1$, for all $k\in\N$, it follows from \eqref{eigen} that $v=0$. But this is a contradiction with \eqref{notnull2}. Therefore $\{u_{n}\}$ is bounded in $H_{0}^{1}(\Omega)$.

\medskip

$(ii)$ If $\lambda_{m+k}(\eta)= 1$, for some $k\in\N$, then \eqref{eigen} implies that $v = e_{m+k}$, where $e_{m+k}$ is some eigenfunction associated to $\lambda_{m+k}(\eta)$. From \eqref{eigen}, it follows also that 
$$
\int_{\Omega}\eta(x)v^{2}dx=\|v\|^{2}\leq \liminf_{n\to\infty}\|v_{n}\|^{2}=1,
$$
that is, $v\in \mathcal{S}_{\mathcal{A}}^{c}$. Suppose that 
\begin{equation}
\int_{\Omega}\eta(x)v^{2}dx<1.
\label{alphainfinity1}
\end{equation}
Since
\begin{equation}\label{equal3}
t_{v_{n}}=\|t_{v_{n}}v_{n}\|=\|u_{n}\|\to\infty,
\end{equation}
by arguing as in \eqref{ahora} and \eqref{salio}, we obtain
\begin{equation}\label{listo}
\int_{\Omega}\left[\frac{F(\|u_{n}\|v_{n})}{(\|u_{n}\|v_{n})^2}\right]v_{n}^{2}dx\to (1/2)\int_{\Omega}\eta(x)v^{2}dx.
\end{equation}
Thus, passing to the limit as $n\to\infty$ in the identity
$$
\Psi(v_{n})=\|u_{n}\|^{2}\left\{\frac{1}{2}-\int_{\Omega}\left[\frac{F(\|u_{n}\|v_{n})}{(\|u_{n}\|v_{n})^2}\right]v_{n}^{2}dx\right\}
$$
and using \eqref{listo}, we conclude that 
$$
\Psi(v_{n})\to\infty,
$$ 
what is a contradiction with \eqref{bound}. Consequently, 
\begin{equation}\label{boundary}
1=\|v\|^{2}=\int_{\Omega}\eta(x)v^{2}dx,
\end{equation}
showing that $v\in\partial\mathcal{S}_{\mathcal{A}}$ and
\begin{equation}\label{norm}
\|v_{n}\|\to \|v\|.
\end{equation}
By using \eqref{weak} and \eqref{norm}, we conclude that $v_{n}\to v$ in $H_{0}^{1}(\Omega)$. By invoking Proposition \ref{main1}, we have
\begin{equation}\label{end}
c\geq \int_{[v\neq 0]}\beta(x) dx,
\end{equation}
with $v=e_{m+k}\in\partial\mathcal{S}_{\mathcal{A}}$. Since $(f_1)-(f_2)$ and \eqref{finito} hold, last equality contradicts Lemma \ref{level}, showing that $\{u_{n}\}$ is bounded. 

\medskip

Since $\{u_{n}\}$ is a bounded sequence, there exists $u\in H_{0}^{1}(\Omega)$ such that $u_{n}\rightharpoonup u$ in $H_{0}^{1}(\Omega)$, up to a subsequence. Thus, to finish the proof, it remains us to prove that $\|u_{n}\|\to \|u\|$. For this, it is sufficient to note that since $\{u_{n}\}$ is a $(PS)_{c}$ sequence, we have
$$
o_{n}(1)+\int_{\Omega}\nabla u_{n}\nabla u dx=\int_{\Omega}f(x, u_{n})u dx.
$$
Passing to the limit as $n\to\infty$ in the previous equality, we get
\begin{equation}\label{equal5}
\|u\|^{2}=\int_{\Omega}f(x, u)u dx.
\end{equation}
Then \eqref{equal5} and Lebesgue's convergence theorem imply that
$$
\|u_n\|^2 = \int_\Omega f(x,u_n)u_ndx = \int_\Omega f(x,u)udx + o_n(1) = \|u\|^2 + o_n(1).
$$
\end{proof}


\section{Multiplicity of solutions}\label{sec:multiplicity}

In this section, among other things, we are going to prove that the problem \eqref{P} has as many pairs of solutions as we want, provided that function $f$ is odd with respect to the second variable and the size of $\lambda_{m}(\eta)$ or $\lambda_{m}(\alpha)$, for $m$ large enough, is controlled from above. 

\medskip

Our main results in this section will be proved through the Krasnosel'skii's genus theory.  Thus, we start by defining some preliminaries notations:
$$
\gamma_{j}:=\left\{B\in \mathcal{E}: B\subset \mathcal{S}_{\mathcal{A}} \ \mbox{and} \ \gamma(B)\geq j\right\},
$$
where 
$$
\mathcal{E}=\{B\subset H_{0}^{1}(\Omega)\backslash\{0\}: B \ \mbox{is closed and $B=-B$}\}
$$ 
and $\gamma:\mathcal{E}\to \mathds{Z}\cup \{\infty\}$ is the Krasnosel'skii's genus function, which is defined by

\medskip

\begin{equation}
\gamma(B)=\left \{ \begin{array}{ll}
n:=\min\{ m\in\N: \ \mbox{there exists an odd map $\varphi\in C(B, \R^{m}\backslash\{0\})$}\},\\
\infty, \mbox{if there exists no map $\varphi\in C(B, \R^{m}\backslash\{0\})$},\\
0, \mbox{if $B=\emptyset$}.
\end{array}\right.
\end{equation}

\medskip

It is important to point that $\gamma_{j}$ is well defined, since $\mathcal{S}_{\mathcal{A}}=-\mathcal{S}_{\mathcal{A}}$.

\medskip

In the sequel we state some standard properties of the genus which will play a role in this work. More information about this
subject can be found, for instance, in \cite{AR} or \cite{Kr}.

\medskip

\begin{lemma}\label{genus}
Let $B$ and $C$ be sets in $\mathcal{E}$.

\medskip

\begin{enumerate}

\item[$(i)$] If $x\neq 0$, then $\gamma(\{x\}\cup \{-x\})=1$;

\item[$(ii)$]  If there exists an odd map $\varphi\in C(B, C)$, then $\gamma(B)\leq \gamma(C)$. In particular, if $B\subset C$ then $\gamma(B)\leq \gamma(C)$.

\item[$(iii)$]  If there exists an odd homeomorphism $\varphi:B\rightarrow C$, then $\gamma(B)=\gamma(C)$. In particular, if $B$ is homeomorphic to the unit sphere in $\R^{n}$, then $\gamma(B)=n$.

\item[$(iv)$]  If $B$ is a compact set, then there exists a neighborhood $K\in \mathcal{E}$ of $B$ such that
$\gamma(B)=\gamma(K)$.

\item[$(v)$]  If $\gamma(C)<\infty$, then
$\gamma(\overline{B\backslash C})\geq \gamma(B)-\gamma(C)$.

\item[$(vi)$]  If $\gamma(A) \geq 2$, then $A$ has infinitely many points.
\end{enumerate}
\end{lemma}

\medskip

From now on, we denote by $s_{m}$ the sum of the dimensions of all eigenspaces $V_{j}$ associated to eigenvalues $\lambda_{j}(\eta)$, where $1\leq j\leq m$. 

\medskip

\begin{lemma}\label{nosso}
Suppose that $f(x,\cdot)$ is odd a.e. in $\Omega$ and satisfies $(f_{2})$. The following statements hold true:
\begin{enumerate}
\item[$(i)$] $\gamma_{s_{m}}\neq \emptyset$;
\item[$(ii)$] $\gamma_{1}\supset\gamma_{2}\supset\ldots\supset\gamma_{s_{m}}$;
\item[$(iii)$] If $\varphi\in C(\mathcal{S}_{\mathcal{A}}, \mathcal{S}_{\mathcal{A}})$ and is odd, then $\varphi(\gamma_j) \subset \gamma_j$, for all $1\leq j\leq s_{m}$;
\item[$(iv)$] If $B\in \gamma_{j}$ and $C\in \mathcal{E}$ with $\gamma(C)\leq s<j\leq s_{m}$, then $\overline{B\backslash C}\in \gamma_{j-s}$.
\end{enumerate}
\end{lemma}

\medskip

\begin{proof}
$(i)$ Let $\mathcal{S}_{s_{m}}$ be the ($s_{m}$-dimensional) unit sphere of $V_{1}\oplus V_{2}\oplus\ldots\oplus V_{m}$. From $(f_{2})$, it is clear that $\mathcal{S}_{s_{m}}\subset \mathcal{S}_{\mathcal{A}}$. Moreover, from Lemma \ref{genus}$(iii)$, we have $\gamma(\mathcal{S}_{s_{m}})=s_{m}$, showing that $\mathcal{S}_{s_{m}}\in \gamma_{s_{m}}$. $(ii)$ It is immediate. $(iii)$ It follows directly from Lemma \ref{genus}$(ii)$. $(iv)$ It is a consequence of Lemma \ref{genus}$(v)$.
\end{proof}

\medskip

Now, for each $1\leq j\leq s_{m}$, we define the following minimax levels
\begin{equation}
c_{j}=\inf_{B\in \gamma_{j}}\sup_{u\in B}\Psi(u).
\end{equation}

\medskip

\begin{lemma}\label{minimax}
Suppose that $f(x,\cdot)$ is odd a.e. in $\Omega$, satisfies $(f_{1})-(f_{2})$ and \eqref{finito}. Then,
$$
0< c_{\mathcal{N}}=c_{1}\leq c_{2}\leq \ldots\leq c_{s_{m}}< \inf_{u\in\partial\mathcal{S}_{\mathcal{A}}}\int_{[u\neq 0]}\beta(x) dx.
$$
\end{lemma}

\medskip

\begin{proof}
$(i)$ First inequality follows from Lemma \ref{lemma3} (see also Remark \ref{rem2}). The equality $c_{\mathcal{N}}=c_{1}$ can be easily derivate from Lemma \ref{genus}$(i)$ and from definition of $c_{1}$. On the other hand, the monotonicity of the levels $c_{j}$ is a consequence of Lemma \ref{nosso}$(ii)$. To prove last inequality, observe that by proof of Lemma \ref{level}, we have
$$
\Psi(u)\leq [1/2\lambda_{1}(\eta-\alpha)]t_{u}^{2}, \ \forall \ u\in \mathcal{S}_{\mathcal{A}}.
$$
Thus, 
\begin{equation}\label{zzz}
c_{s_{m}}\leq\max_{u\in \mathcal{S}_{s_{m}}}\Psi(u)\leq [1/2\lambda_{1}(\eta-\alpha)]\tau_{m}^{2},
\end{equation}
where $\tau_{m}$ was defined in Proposition \ref{proposition2}. The result follows now from Lemma \ref{limita}, \eqref{finito} and \eqref{zzz}.
\end{proof}

\medskip

Next proposition is crucial to ensure the multiplicity of solutions.

\medskip

\begin{proposition}\label{multmais}
Suppose that $f(x,\cdot)$ is odd a.e. in $\Omega$ and satisfies $(f_{1})-(f_{2})$. If $c_{j}=\ldots=c_{j+p}\equiv c$, $j+p\leq s_{m}$ and $(NRC)$, or $(RC)$ and \eqref{finito}, occurs then $\gamma(K_{c})\geq p+1$, where $K_{c}:=\{v\in \mathcal{S}_{\mathcal{A}}: \Psi(v)=c\ \ \mbox{and} \ \Psi'(v)=0\}$.
\end{proposition}

\medskip

\begin{proof}
Suppose that $\gamma(K_{c})\leq p$. It follows from Proposition \ref{main2} and Lemma \ref{minimax}$(i)$ that $K_{c}$ is a compact set. Thus, by Lemma \ref{genus}$(iv)$, there exists a compact neighborhood $K\in \mathcal{E}$ of $K_{c}$ such that $\gamma(K)\leq p$. Defining $M:=K\cap \mathcal{S}_{\mathcal{A}}$, we derive from Lemma \ref{genus}$(ii)$ that $\gamma(M)\leq p$. Despite the noncompleteness of $\mathcal{S}_{\mathcal{A}}$ we still can use Theorem 3.11 in \cite{Str} (see also Remark 3.12 in \cite{Str}) to ensure the existence of an odd homeomorphisms family $\eta(., t)$ of $\mathcal{S}_{\mathcal{A}}$ such that, for each $u\in \mathcal{S}_{\mathcal{A}}$, 
\begin{equation}\label{iguald}
\eta(u, 0)=u
\end{equation} 
and
\begin{equation}\label{traj}
t\mapsto\Psi(\eta(u, t)) \ \mbox{is non-increasing}.
\end{equation}

\medskip

Observe that, although $\mathcal{S}_{\mathcal{A}}$ is non-complete, from Proposition \ref{main1}, Lemma \ref{minimax} and \eqref{traj}, for $\varepsilon>0$ small enough, map $\eta(u, .)$ is well defined in $[0,\infty)$, for each $u\in \Psi_{c_{s_{m}}+\varepsilon}=\{u\in \mathcal{S}_{\mathcal{A}}: \Psi(u)\leq c_{s_{m}}+\varepsilon\}$. Indeed, suppose by contradiction that $\eta(u, t_{0})\in \partial\mathcal{S}_{\mathcal{A}}$ for some $u\in \Psi_{c_{s_{m}}+\varepsilon}$ and $t_{0}>0$, where $\varepsilon\in (0, \inf_{u\in \partial\mathcal{S}_{\mathcal{A}}}\int_{[u\neq 0]}\beta(x) dx- c_{s_{m}})$. Then, by \eqref{iguald}, Proposition \ref{main1} and Lemma \ref{minimax}
$$
\Psi(\eta(u, 0))=\Psi(u)\leq c_{s_{m}}+\varepsilon<\int_{[\eta(u, t_{0})\neq 0]}\beta(x) dx\leq \liminf\operatorname{ess} \limits_{t\to t_{0}}\Psi(\eta(u, t)).
$$
Thus, there exists $0<t_{\ast}<t_{0}$, such that $\eta(u, t_{\ast})\in \mathcal{S}_{\mathcal{A}}$ and
$$
\Psi(\eta(u, 0))<\Psi(\eta(u, t_{\ast})).
$$
However, last inequality contradicts \eqref{traj}. Thus, third claim of Theorem 3.11 in \cite{Str} holds, namely,
\begin{equation}\label{def}
\eta(\Psi_{c+\varepsilon}\backslash M, 1)\subset \Psi_{c-\varepsilon}.
\end{equation}
Let us choose $B\in \gamma_{j+p}$ such that $\sup_{B}\Psi\leq c+\varepsilon$. From Lemma \ref{nosso}$(iv)$, $\overline{B\backslash M}\in\gamma_{j}$. It follows again from Lemma \ref{nosso}$(iii)$ that $\eta(\overline{B\backslash M}, 1)\in \gamma_{j}$. Therefore, by \eqref{def} and definition of c, we have
$$
c\leq \sup_{\eta(\overline{B\backslash M}, 1)}\Psi\leq c-\varepsilon,
$$
that is a contradiction. Then $\gamma(K_{c})\geq p+1$.
\end{proof}

\medskip

We are now ready to prove the following two multiplicity results.

\medskip

\begin{theorem}\label{teo1}
Suppose that $f$ satisfies $(f_{1})-(f_{2})$ and $(NRC)$, or $(RC)$ and \eqref{finito}. The following statements hold:\\
$(i)$ If $m=1$, then there exists a signed mountain pass ground-state solution for problem \eqref{P}.\\
$(ii)$ If $f(x,\cdot)$ is odd a.e. in $\Omega$, then problem \eqref{P} has at least $\chi(\eta)$ pairs of nontrivial solutions with positive energy.
\end{theorem}

\begin{proof}

$(i)$ Let $\{u_{n}\}\subset \mathcal{N}$ be such that $I(u_{n})\to c_{\mathcal{N}}$. Remember that $v_{n}:=u_{n}/\|u_{n}\|\in \mathcal{S}_{\mathcal{A}}$ (see Proposition \ref{proposition2}$(A_{3})$) and 
\begin{equation}\label{aca}
\Psi(v_{n})\to c_{\mathcal{N}}.
\end{equation}

\medskip

We are going to prove the existence of a sequence $\{\hat{v}_{n}\}$ which is a $(PS)_{c_{\mathcal{N}}}$ sequence for the functional $\Psi$. For this, let $\overline{\mathcal{S}}_{\mathcal{A}}$ be the closure of $\mathcal{S}_{\mathcal{A}}$ in $H_{0}^{1}(\Omega)$ and consider the map $\Upsilon:\overline{\mathcal{S}}_{\mathcal{A}}\to\R\cup\{\infty\}$ defined by 

\medskip

\begin{equation}
\Upsilon(u)=\left \{ \begin{array}{ll}
\Psi(u) & \mbox{if $u\in \mathcal{S}_{\mathcal{A}}$,}\\
\int_{[u\neq 0]}\beta(x) dx & \mbox{if $u\in\partial\mathcal{S}_{\mathcal{A}}$.}
\end{array}\right.
\end{equation}

\medskip

It follows from Lemma \ref{level} that $c_{\mathcal{N}}=\inf_{u\in \overline{\mathcal{S}}_{\mathcal{A}}}\Upsilon(u)$. 

\medskip

Let us show that $\Upsilon$ is lower semicontinuous. In fact, let $\{u_{n}\}\subset \overline{\mathcal{S}}_{\mathcal{A}}$ and $u\in \overline{\mathcal{S}}_{\mathcal{A}}$ be such that $u_{n}\to u$ in $H_{0}^{1}(\Omega)$. If $u\in \mathcal{S}_{\mathcal{A}}$ then, for $n$ large enough, $\Upsilon(u_{n})=\Psi(u_{n})$ and
$$
\Upsilon(u_{n})=\Psi(u_{n})\to\Psi(u)=\Upsilon(u),
$$
because $\Psi$ is continuous. On the other hand, if $u\in \partial\mathcal{S}_{\mathcal{A}}$, we have two cases to consider. If there exists a subsequence $\{u_{n}\}\subset \mathcal{S}_{\mathcal{A}}$ then, by Proposition \ref{main1},
$$
\Upsilon(u)=\int_{[u\neq 0]}\beta(x) dx\leq \liminf\limits_{n\to\infty}\Psi(u_{n})=\liminf\limits_{n\to\infty}\Upsilon(u_{n}).
$$
If there exists a subsequence $\{u_{n}\}\subset \partial\mathcal{S}_{\mathcal{A}}$ then, from $\beta(x)\geq 0$ and
$$
\chi_{[u_n\neq 0]}(x)\to 1 \ \mbox{a.e. in $[u\neq 0]$},
$$ 
we have
$$
\Upsilon(u)=\int_{[u\neq 0]}\beta(x) dx\leq \liminf\limits_{n\to\infty}\int_{[u_{n}\neq 0]}\beta(x) dx=\liminf\limits_{n\to\infty}\Upsilon(u_{n}).
$$
This, in turn, shows that $\Upsilon$ is a lower semicontinuous map.

\medskip

Since $\overline{\mathcal{S}}_{\mathcal{A}}$ is a complete metric space (with metric provided by the norm of $H_{0}^{1}(\Omega)$) and $\Upsilon$ is bounded from below (see Lemma \ref{lemma3}), it follows from Theorem 1.1 in \cite{Ek} that for each $\varepsilon, \lambda>0$ 
small enough and $u\in \Upsilon^{-1}[c_\mathcal{N},c_\mathcal{N} + \varepsilon]$ there exists $v\in \overline{\mathcal{S}}_{\mathcal{A}}$ such that
$$
c_{\mathcal{N}}\leq \Upsilon(v)\leq \Upsilon(u), \  \|u-v\|\leq \lambda \ \mbox{and} \ \Upsilon(w)>\Upsilon(v)-(\varepsilon/\lambda)\|v-w\|, \ \forall \ w\neq v.
$$

\medskip

On the other hand, it follows from Lemma \ref{level} that, for $\varepsilon$ small enough,
\begin{equation}\label{equali}
\Upsilon^{-1}[c_\mathcal{N},c_\mathcal{N} + \varepsilon]=\Psi^{-1}[c_\mathcal{N},c_\mathcal{N} + \varepsilon], v\in \mathcal{S}_{\mathcal{A}} \ \mbox{and} \ \Upsilon(v)=\Psi(v).
\end{equation} 
Passing to a subsequence, it follows from \eqref{aca} that we can choose $u=v_{n}$, $\varepsilon=1/n^{2}$ and $\lambda=1/n$ to get
$\widehat{v}_{n}\in \mathcal{S}_{\mathcal{A}}$, satisfying
\begin{equation}\label{PS1}
\Psi(\widehat{v}_{n})\to c_{\mathcal{N}}, \ \|v_{n}- \widehat{v}_{n}\|\to 0
\end{equation}
and
\begin{equation}\label{deriv}
\Upsilon(w)>\Psi(\widehat{v}_{n})-(1/n)\|\widehat{v}_{n}-w\|, \ \forall \ w\neq \widehat{v}_{n}.
\end{equation}

\medskip

Let $\gamma_{n}: (-\delta_{n}, \delta_{n})\to \mathcal{S}_{\mathcal{A}}$ be a differentiable curve, with $\delta_{n}>0$ small enough, such that 
$\gamma_{n}(0)=\widehat{v}_{n}$ and $\gamma'_{n}(0)=z\in T_{\widehat{v}_{n}}(\mathcal{S}_{\mathcal{A}})$. 
Choosing $w=\gamma_{n}(t)$, it follows from $(\ref{deriv})$ that
\begin{equation}\label{deriv1}
-[\Psi(\gamma_{n}(t))-\Psi(\gamma_{n}(0))]<(1/n)\|\gamma_{n}(t)-\gamma_{n}(0)\|.
\end{equation}
By Mean Value Theorem, there exists $c\in (0, t)$, such that
\begin{equation}
 \|\gamma_{n}(t)-\gamma_{n}(0)\|
\leq \|  \gamma'_{n}(c)\||t|. 
\label{deriv2}
\end{equation}

\medskip

It follows from \eqref{deriv1} and \eqref{deriv2} that, multiplying both sides by $1/t$ and passing to the limit of $t\rig 0^{+}$, we get
\begin{equation}\label{deri1}
-\Psi'(\widehat{v}_{n})z\leq \frac{1}{n}\|z\|.
\end{equation}
Since $z\in  T_{\widehat{v}_{n}}(\mathcal{S}_{\mathcal{A}})$ is arbitrary, by linearity, we have
$$
|\Psi'(\widehat{v}_{n})z|\leq \frac{1}{n}\|z\|.
$$
Therefore,
\begin{equation}\label{PS3}
\|\Psi'(\widehat{v}_{n})\|_{\ast}\to 0,
\end{equation}
as $n\to\infty$, and, by \eqref{PS1}, we conclude that $\{v_{n}\}$ is a $(PS)_{c_{\mathcal{N}}}$ sequence for $\Psi$. It follows from Lemma \ref{level} and Proposition \ref{main2} that there exists $v\in \mathcal{S}_{\mathcal{A}}$ such that, passing to a subsequence, $v_{n}\to v$ in $H_{0}^{1}(\Omega)$. Thus $\Psi'(v)=0$ and $\Psi(v)=c_{\mathcal{N}}$. Defining $u:=m(v)\in\mathcal{N}$ and using Proposition \ref{proposition3}$(iv)$, we conclude that $I'(u)=0$ and $I(u)=c_{\mathcal{N}}$.

\medskip

To show that $u$ does not change sign, observe that if $u^\pm \neq 0$, then $u^{\pm}\in \mathcal{N}$. Thus,
\begin{equation}\label{sign}
c_{\mathcal{N}}=I(u)=I(u^{+})+I(u^{-})\geq 2c_{\mathcal{N}},
\end{equation}
which is a clear contradiction. Therefore, it follows that either $u^+ = 0$ or $u^- = 0$ and, consequently, $u$ is a signed solution.

\medskip

$(ii)$ First of all, note that the levels $0<c_{j}<\infty$ are critical levels of $\Psi$. In fact, suppose by contradiction that $c_{j}$ is regular for some $j$. Invoking Theorem 3.11 in \cite{Str}, with $\beta=c_{j}$, $\overline{\varepsilon}=1$, $N=\emptyset$, there exist $\varepsilon>0$ and a family of odd homeomorphisms $\eta(., t)$ satisfying the properties of the referred theorem. Choosing $B\in \gamma_{j}$ such that $\sup_{B}\Psi<c_{j}+\varepsilon$ and arguing as in the proof of Proposition \ref{multmais} we get a contradiction. 

\medskip

Finally, if levels $c_{j}$, $1\leq j\leq s_{m}$, are different from each other, it follows from Proposition \ref{proposition3}$(iv)$ that the result is proved. On the other hand, if $c_{j}=c_{j+1}\equiv c$ for some $1\leq j\leq s_{m}$, it follows from Proposition \ref{multmais} that $\gamma(K_{c})\geq 2$. Combining last inequality with Lemma \ref{genus}$(vi)$ and Proposition \ref{proposition3}$(iv)$, we conclude that \eqref{P} has infinitely many pairs of nontrivial solutions. The result now is proved.

\end{proof}

\begin{theorem}\label{teo3}
Suppose that $f(x,\cdot)$ is odd a.e. in $\Omega$ and satisfies $(f_{1}')-(f_{2}')$. The following statements hold:\\
$(i)$ If $m=1$, then problem \eqref{P} has a nontrivial solution;\\
$(ii)$ If $f(x,\cdot)$ is odd a.e. in $\Omega$, then problem \eqref{P} has at least $\chi(\alpha)$ pairs of nontrivial solutions with negative energy.
\end{theorem}

\begin{proof}

Since $(f_1')$ and $(f_2')$ are satisfied, we are going to prove that, in any case, $I$ is coercive and bounded from below. For that, let $\{u_{n}\}\subset H_{0}^{1}(\Omega)$ be a sequence with $\|u_{n}\|\to\infty$. If $v_{n}:=u_{n}/\|u_{n}\|$, then
$$
\frac{I(u_{n})}{\|u_{n}\|^{2}}=\frac{1}{2}-\int_{\Omega}\left[\frac{F(x, u_{n})}{\|u_{n}\|^{2}}\right]dx.
$$
Observe that, up to a subsequence, 
$$
v_{n}\rightharpoonup v \ \mbox{in $H_{0}^{1}(\Omega)$},
$$
$$
\chi_{[v_{n}\neq 0]}(x)\to 1 \ \mbox{a.e. in $[v\neq 0]$},
$$
$$
\chi_{[v_{n}\neq 0]}(x)\to 0 \ \mbox{a.e. in $[v= 0]$}
$$
and
$$
\int_{\Omega}\left[\frac{F(x, u_{n})}{\|u_{n}\|^{2}}\right]dx=\int_{[v\neq 0]}\left[\frac{F(x, \|u_{n}\|v_{n})}{\|u_{n}\|^{2}v_{n}^{2}}\right]v_{n}^{2}\chi_{[v_{n}\neq 0]}(x)dx+\int_{[v=0]}\left[\frac{F(x, \|u_{n}\|v_{n})}{\|u_{n}\|^{2}v_{n}^{2}}\right]v_{n}^{2}\chi_{[v_{n}\neq 0]}(x)dx.
$$
Thus, from ($f_1'$) and Lebesgue Dominated Convergence Theorem 
$$
\int_{\Omega}\left[\frac{F(x, u_{n})}{\|u_{n}\|^{2}}\right]dx\to \frac{1}{2}\int_{\Omega}\eta(x)v^{2} dx.
$$
Therefore, by weighted Poincar\'e inequality (see Proposition 1.10 in \cite{DeFig})
$$
\frac{I(u_{n})}{\|u_{n}\|^{2}}\to \frac{1}{2}\left(1-\int_{\Omega}\eta(x)v^{2} dx\right)\geq \frac{1}{2}\left(1-\frac{1}{\lambda_1(\eta)}\|v\|^{2}\right).
$$
Finally, convergence \eqref{6} and the weak lower semicontinuity of the norm imply that $\|v\|\leq 1$. Consequently,
$$
\frac{I(u_{n})}{\|u_{n}\|^{2}}\to \frac{1}{2}\left(1-\int_{\Omega}\eta(x)v^{2} dx\right)\geq \frac{1}{2}\left(1-\frac{1}{\lambda_1(\eta)}\right)>0,
$$
where last inequality comes from $(f_2')$. This proves that $I$ is coercive. Since $f$ is Carathe\'odory, it follows that $I$ is weakly lower semicontinuous. Consequently, $I$ is bounded from below and it has a minimum point $u_{\ast}\in H_{0}^{1}(\Omega)$ which is a nontrivial solution of \eqref{P}. In fact, if $e_1$ is the positive eigenfunction (normalized in $H_{0}^{1}(\Omega)$) associated to the first eigenvalue $\lambda_1(\alpha)$, then
$$
I(te_1)=\frac{1}{2}-\int_{\Omega}\left[\frac{F(x, te_1)}{(te_1)^{2}}\right]e_1^{2}dx, \ \forall \ t>0.
$$

\medskip

It follows from $(f_1')-(f_{2}')$ and Lebesgue Dominated Convergence Theorem that
\begin{equation}\label{3}
\lim_{t\to 0^{+}}\frac{I(te_1)}{t^{2}}= \frac{1}{2}\left[1-\int_{\Omega}\alpha(x)e_{1}^{2}dx\right]=\frac{1}{2}\left[1-\frac{1}{\lambda_1(\alpha)}\right]<0.
\end{equation}
This shows that there exist $\varepsilon, t_{\ast}>0$ small enough such that
$$
I(u_{\ast})\leq I(te_{1})\leq -\varepsilon t_{\ast}^{2}.
$$
This proves item $(i)$.

\medskip

$(ii)$ Now, observe that for each $u\in \mathcal{S}_{\chi(\alpha)}$ and $t>0$,
$$
\frac{I(tu)}{t^{2}}= \frac{1}{2}-\int_{[u\neq 0]}\left[\frac{F(x, tu)}{(tu)^{2}}\right]u^{2}dx.
$$
It follows from $(f_1')$ and Lebesgue Dominated Convergence Theorem that
\begin{equation}\label{3}
\lim_{t\to 0^{+}}\frac{I(tu)}{t^{2}}= \frac{1}{2}\left[1-\int_{\Omega}\alpha(x)u^{2}dx\right].
\end{equation}

\medskip

Let $dim V_{\lambda_{i}(\alpha)}$ be the dimension of the eigenspace $V_{\lambda_{i}(\alpha)}$ and $\{e_{ij}\}\subset \oplus_{k=1}^{m}V_{\lambda_k(\alpha)}$ the associated orthonormal basis of eigenfunctions. Since $u$ can be written as
$$
u=\sum_{i=1}^{m}\sum_{j=1}^{dim V_{\lambda_{i}(\alpha)}}u_{ij}e_{ij},
$$ 
with
$$
\sum_{i=1}^{m}\sum_{j=1}^{dim V_{\lambda_{i}(\alpha)}}u_{ij}^{2}=1,
$$
we conclude that
\begin{equation}\label{4}
\lim_{t\to 0^{+}}\frac{I(tu)}{t^{2}}= \frac{1}{2}\left[1-\sum_{i=1}^{m}\sum_{j=1}^{dim V_{\lambda_{i}(\alpha)}}u_{ij}^{2}\int_{\Omega}\alpha(x)e_{ij}^{2} dx\right]=\frac{1}{2}\left[1-\sum_{i=1}^{m}\sum_{j=1}^{dim V_{\lambda_{i}(\alpha)}}\frac{u_{ij}^{2}}{\lambda_{i}(\alpha)}\right].
\end{equation}

\medskip

Since $\lambda_{i}(\alpha)\leq \lambda_{m}(\alpha)$ for all $i\in\{1, \ldots, m\}$ and $u\in \mathcal{S}_{\chi(\alpha)}$,
\begin{equation}\label{5}
\lim_{t\to 0^{+}}\frac{I(tu)}{t^{2}}\leq \frac{1}{2}\left[1-\frac{1}{\lambda_{m}(\alpha)}\sum_{i=1}^{m}\sum_{j=1}^{dim V_{\lambda_{i}(\alpha)}}u_{ij}^{2}\right]=\frac{1}{2}\left[1-\frac{1}{\lambda_{m}(\alpha)}\right]<0,
\end{equation}
where the last inequality comes from $(f_2')$. Therefore, there exist $\varepsilon, \delta>0$ such that 
$$
I(tu)=(I(tu)/t^{2})t^{2}\leq -\varepsilon t^{2},
$$
for all $0<t<\delta$ and $u\in \mathcal{S}_{\chi(\alpha)}$. Fixing $0<t_{\ast}<\delta$, we have
$$
\sup_{w\in t_{\ast}\mathcal{S}_{\chi}}I(w)<0.
$$
Since $I$ is coercive, it is standard to prove that it satisfies the $(PS)_{c}$ condition. Finally, as $I$ is an even $C^1$-functional and $I(0)=0$, it follows from Theorem 9.1 in \cite{Rab}, that $I$ has at least $\chi(\alpha)$ pairs of critical points.

\end{proof}


\section{On the assumption \eqref{finito}}\label{sec:assump}

In this section we are interested in providing a concrete example of function $f$ which satisfies hypothesis \eqref{finito} considered in Theorem \ref{teo1}, but does not satisfy assumption $(fF)$. In order to fix some ideas, let us consider $\eta>\lambda_{m}$ arbitrarily fixed, where $\lambda_{m}$ denotes the first eigenvalue of the laplacian operator with Dirichlet boundary condition. Let also $\mathcal{A}$, the open set
$$
\mathcal{A}=\mathcal{A}_{\eta}=\{u\in H_{0}^{1}(\Omega): \|u\|^{2}<\eta\int_{\Omega}u^{2}dx\}.
$$ 
Let $u_{\ast}\in \mathcal{S}_{\chi(\eta)}\cap L^{\infty}(\Omega)\subset \mathcal{S}_{\mathcal{A}}$ and $\theta:=|u_{\ast}|_{\infty}$. We are going to consider the problem
\begin{equation}\label{CP}\tag{CP}
\left \{ \begin{array}{ll}
-\Delta u =  f(u)& \mbox{in $\Omega$,}\\
u=0 & \mbox{on $\partial\Omega$,}
\end{array}\right.
\end{equation}
with $\Omega\subset \R^{N}$, $N\geq 1$, a bounded smooth domain, 
$$
f(t)=\left \{ \begin{array}{ll} t|t| & \mbox{if $|t|\leq \theta$,}\\
\eta\frac{t^{5}}{a+t^{4}} & \mbox{if $|t|>\theta$,}
\end{array}\right.
$$
and  
\begin{equation}\label{conta5}
a=\theta^{3}(\eta-\theta)
\end{equation} 
is such that $f$ is a continuous function (and non-differentiable in $-\theta$ and $\theta$). 

\medskip

Observe that for $\eta$ large enough, we still have $u_{\ast}\in \mathcal{S}_{\mathcal{A}}$ because $\eta_{1}>\eta_{2}>\lambda_1$ implies $\mathcal{A}_{\eta_{2}}\subset\mathcal{A}_{\eta_{1}}$. Moreover, for $\eta$ large enough, $a$ is a positive constant. 

\medskip

Some simple calculations show us that $f$ is an odd function satisfying $(f_{1})-(f_{2})$, with $\alpha(x)=0$, $\eta(x)=\eta$, $\lambda_{1}(\eta-\alpha)=\lambda_1/\eta$ and $\beta(x)=\beta$, where
\begin{eqnarray*}
\beta&=&\lim_{|t|\to\infty}\left[\frac{1}{2}f(t)t-F(t)\right]\\
&=&\frac{\eta}{2}\left(\theta^{2}-\sqrt{a}\arctan\left(\frac{\theta^{2}}{\sqrt{a}}\right)\right)-\frac{\theta^{3}}{3}+\lim_{|t|\to\infty}\left\{\frac{1}{2}\left[\frac{\eta t^{6}}{a+t^{4}}\right]-\left[\frac{\eta}{2}\left(t^{2}-\sqrt{a}\arctan\left(\frac{t^{2}}{\sqrt{a}}\right)\right)\right]\right\}\\
&=&\frac{\eta}{2}\theta^{2}-\frac{\eta\sqrt{a}}{2}\arctan\left(\frac{\theta^{2}}{\sqrt{a}}\right)-\frac{\theta^{3}}{3}+\frac{\pi\eta\sqrt{a}}{4}.
\end{eqnarray*}
Thus, hypothesis \eqref{fF} in \cite{LZ} is not verified. Note that $f(t)/t\to 0$ faster than $t^{2}\to\infty$ as $|t|\to\infty$, that is,
$$
\lim_{|t|\to\infty}[\eta-f(t)/t]t^{2}=\eta\lim_{|t|\to\infty}\left[1-\frac{t^{4}}{a+t^{4}}\right]t^{2}=0.
$$

\medskip

Let $I: H_{0}^{1}(\Omega)\to\R$ be the energy functional of \eqref{CP}. Since $f$ satisfies $(f_{1})-(f_{2})$, Proposition \ref{proposição1} holds for $I$. That is, for each $u\in \mathcal{A}$, there exists a unique $t_{u}>0$ such that $t_{u}u\in\mathcal{N}$, where $\mathcal{N}$ is the Nehari manifold associated to $I$. On the other hand, from definition of $f$, it is clear that, for $\eta$ large and $0<t<1$, 
$$
I(tu_{\ast})= \frac{t^{2}}{2}\|u_{\ast}\|^{2}-\frac{t^{3}}{3}\int_{\Omega}|u_{\ast}|^{3}dx
$$ 
and 
\begin{equation}\label{derivs}
\alpha_{u_{\ast}}'(t)=J'(t u_{\ast})(u_{\ast})=t-\int_{\Omega}f(x, t u_{\ast})u_{\ast} dx=t-t^{2}\int_{\Omega}|u_{\ast}|^{3} dx.
\end{equation}
Since $u_{\ast}\in \mathcal{S}_{\mathcal{A}}$, by H\"{o}lder's inequality, we get
$$
1<\eta\int_{\Omega}u_{\ast}^{2}dx\leq \eta|\Omega|^{1/3}\left(\int_{\Omega}|u_{\ast}|^{3}dx\right)^{2/3}.
$$
Showing that
\begin{equation}\label{conta2}
\int_{\Omega}|u_{\ast}|^{3}dx>\frac{1}{\eta^{3/2}|\Omega|^{1/2}}.
\end{equation}
It follows from \eqref{conta2} that
\begin{equation}\label{derivs2}
0<t_{\ast}=\frac{1}{\int_{\Omega}|u_{\ast}|^{3}dx}<\frac{1}{\eta^{3/2}|\Omega|^{1/2}}<1,
\end{equation}
for $\eta$ large. Therefore, from \eqref{derivs} and \eqref{derivs2}, 
$$
\alpha_{u_{\ast}}'(t_{\ast})=0,
$$
and by uniqueness 
\begin{equation}\label{derivs3}
t_{u_{\ast}}=t_{\ast}.
\end{equation}

\medskip

In order to apply Theorem \ref{teo1} to \eqref{CP}, it is enough to prove that assumption \eqref{finito} holds. In the case of problem \eqref{CP}, it is equivalent to
$$
\frac{\eta}{2}\theta^{2}-\frac{\eta\sqrt{a}}{2}\arctan\left(\frac{\theta^{2}}{\sqrt{a}}\right)-\frac{\theta^{3}}{3}+\frac{\pi\eta\sqrt{a}}{4}>\frac{\eta^{2}\tau_{m}^{N/2}}{2S(\Omega)^{N/2}},
$$
or yet, 
\begin{equation}\label{eita}
S(\Omega)^{N/2}\left[\frac{\theta^2}{\eta}-\frac{\sqrt{a}}{\eta}\arctan\left(\frac{\theta^{2}}{\sqrt{a}}\right)-\frac{2\theta^{3}}{3\eta^2}+\frac{\pi\sqrt{a}}{2\eta}\right]>\tau_{m}^{N/2},
\end{equation}
where $\tau_{m}$ is the number defined in Proposition \ref{proposition2}. In order to prove that \eqref{eita} holds, it is enough to show that
\begin{equation}\label{conta3}
S(\Omega)^{N/2}\left[\frac{\theta^2}{\eta}-\frac{\sqrt{a}}{\eta}\arctan\left(\frac{\theta^{2}}{\sqrt{a}}\right)-\frac{2\theta^{3}}{3\eta^2}+\frac{\pi\sqrt{a}}{2\eta}\right]>t_{u_{\ast}}^{N/2}.
\end{equation}
However, it is a straightforward consequence of \eqref{conta5} that, for $\eta$ large enough, we have
\begin{equation}\label{conta8}
\frac{\pi\sqrt{a}}{2\eta}-\frac{\sqrt{a}}{\eta}\arctan\left(\frac{\theta^{2}}{\sqrt{a}}\right)-\frac{2\theta^{3}}{3\eta^2}\geq \frac{C_{1}}{\eta^{1/2}}- \frac{C_{2}}{\eta^{1/2}}\arctan\left(\frac{C_3}{\eta^{1/2}}\right)- \frac{C_{4}}{\eta^{2}}\geq \frac{C_5}{\eta^{1/2}},
\end{equation}
for some positive constants $C_1, C_2, C_3, C_4$ and $C_5$. It follows from \eqref{conta3} and \eqref{conta8} that, for $\eta$ large enough
\begin{equation}\label{conta9}
S(\Omega)^{N/2}\left[\frac{\theta^2}{\eta}-\frac{\sqrt{a}}{\eta}\arctan\left(\frac{\theta^{2}}{\sqrt{a}}\right)-\frac{2\theta^{3}}{3\eta^2}+\frac{\pi\sqrt{a}}{2\eta}\right]\geq \frac{C_5}{\eta^{1/2}}.
\end{equation}
Thus, from \eqref{derivs2}, \eqref{derivs3} and \eqref{conta9}, to prove \eqref{eita}, it is enough to show that
$$
 \frac{C_5}{\eta^{1/2}}>\frac{1}{\eta^{3N/4}|\Omega|^{N/4}},
$$
for $\eta$ large. Since $N\geq 1$, last equality ever occurs. Showing that $f$ satisfies $(\beta)$. Consequently, Theorem \ref{teo1} can be applied to conclude the multiplicity of solutions for problem \eqref{CP}.


\end{document}